\documentclass[a4paper,11pt]{amsart}

\usepackage{a4wide}

\usepackage{graphicx}
\usepackage{pictex}
\usepackage{epstopdf}
\usepackage{etex}
\usepackage{hyperref}
\usepackage{color}
\usepackage{amsthm}

\usepackage{latexsym}
\usepackage{autograph,epic,latexsym,bezier,amsbsy,color,enumerate,amsfonts,amsmath,amscd,amssymb}
\usepackage[latin1]{inputenc}

 \newlength\headseptemp

\numberwithin{equation}{section}    % Gleichungen mit Sections nummeriert
\theoremstyle{plain}
\newtheorem{Theorem}{Theorem}
\newtheorem{Proposition}[Theorem]{Proposition}

\newtheorem{Lemma}[Theorem]{Lemma}
\theoremstyle{definition}

\theoremstyle{remark}

\newcommand{\NN}{\mathbb{N}}

\newcommand{\ZZ}{\mathbb{Z}}

\newcommand{\complexity}{\mathcal{C}}

\begin{document}

\title[The subshift of Grigorchuk's group]{Combinatorics of  the subshift associated to Grigorchuk's group}

\author{Rostislav Grigorchuk}
\address{Mathematics Department, Texas A\&M University, College Station, TX 77843-3368, USA}
\email{grigorch@math.tamu.edu}

\author{Daniel Lenz}
\address{Mathematisches Institut \\Friedrich Schiller
Universit{\"a}t Jena \\07743 Jena, Germany }
\email{daniel.lenz@uni-jena.de}

\author{Tatiana Nagnibeda}
\address{Section de Math\'{e}matiques, Universit\'{e} de Gen\`{e}ve, 2-4, Rue du
Li\`{e}vre, Case Postale 64 1211 Gen\`{e}ve 4, Suisse}
\email{Tatiana.Smirnova-Nagnibeda@unige.ch}

\date{\today}

\begin{abstract} We study combinatorial properties of the subshift induced by the substitution that describes Lysenok's presentation of Grigorchuk's group of intermediate growth by generators and relators. This subshift has recently appeared in two different contexts: on one hand, it allowed to embed Grigorchuk's group in a topological full group, and on the other hand, it was useful in the spectral theory of Laplacians on the
associated Schreier graphs.
\end{abstract}

\maketitle

\begin{center}
\textit{In memory of Dmitry Victorovich Anosov}
\end{center}

\bigskip

%\begin{center}
%{\footnotesize{\bf Mathematics Subject Classification (2010):}
%82B20, 05A15, 20E08.\footnote{This research has been supported by
%the Swiss National Science Foundation.}}
%\end{center}

%\tableofcontents

%\unitlength=0,4mm \textwidth = 16.00cm \textheight = 22.00cm
%\oddsidemargin= 0.12in \evensidemargin = 0.12in
%\setlength{\parindent}{8pt} \setlength{\parskip}{5pt plus 2pt minus
%1pt} \setloopdiam{10}\setprofcurve{7}

%\section*{Introduction}
Substitutional dynamical systems constitute
an important class studied in symbolic dynamics. Such systems are
defined by a substitution over the underlying alphabet. They appear
naturally in various branches of mathematics and applications. In
particular, certain substitutional systems provide important models
for the theory of 'aperiodic order',  and the spectral theory of the
associated Schr\"odinger operators becomes a major tool in
understanding the quantum mechanics of quasi-crystals
\cite{BGr,KLS}. Recently it was discovered that substitutional
subshifts are also useful in the study of groups of intermediate
growth \cite{MB,GLN, GLN-survey,Nek}.

This note is devoted to the study of a particular substitution associated to the first group of intermediate growth
%that is a relative of the substitution used by Lysenok \cite{Lys}
%for getting a presentation of the group $\mathcal{J}$ of
%intermediate growth between polynomial and exponential
constructed by the first author in 1980 in \cite{Gri80} and
generally known as Grigorchuk's group. \footnote{This is how we will refer to it in spite of the first author's reluctance} The
remarkable properties of this group described in \cite{Gri84} were first reported
at Anosov's seminar in the Moscow State University in 1982-83.
The group, defined by its action by automorphisms on the rooted binary tree, is 3-generated but not finitely presented, that is, does not admit a presentation with finitely many relations. However,
 Lysenok found in \cite{Lys} the following recursive presentation of this group by generators and relators
$$\mathcal{J} =
\langle a,b,c,d \mid 1 = a^2 = b^2 = c^2 = d^2 = bcd = \kappa^k
((ad)^4)
 =\kappa^k ((adacac)^4), k = 0,1,2,....\rangle$$
 with
$$\kappa :\{a,b,c,d\}\longrightarrow \{a,b,c,d\}^*$$
given by $\kappa (a) = aca, \kappa (b) = d, \kappa (c) = b, \kappa
(d) = c$.

In \cite{GLN} the authors showed that the substitution appearing in
Lysenok's presentation can be used to determine the spectral type of
the discrete Laplacian on the Schreier graphs naturally associated
with $\mathcal{J}$ via the action of $\mathcal{J}$ on the boundary
of the rooted binary tree. Another interesting fact observed in
\cite{MB} (that also follows from \cite{GLN}), is that $\mathcal{J}$
embeds into the topological full group of a related substitutional
subshift. Various properties of this subshift  were described in
\cite{GLN,GLN-survey}. In particular, it was shown that the subshift
is linearly repetitive. Linear repetitivity has various consequences
including linear bounds on the growth of the word  complexity
function. The purpose of this note is to give a more refined
treatment of this  complexity function.

The word complexity  function $\mathcal{C}$   counts,  for each $n$,
the number of words of length $n$ that appear as subwords in
sequences from the subshift (see below for formal definition). It is
an important combinatorial characteristic of the dynamical system.
For example, the topological entropy of the system can be computed
as
$$h:=\lim_{n\to \infty} \frac{ \ln \mathcal{C}(n)}{n}$$
(see e.g. \cite{Wal}). Hence, linearly repetitive subshifts have
zero entropy.  The main results of the note are the following.

Consider the alphabet
$\mathcal{A} = \{a,x,y,z\}$ and let  $\tau$ be the substitution
mapping $a\mapsto a x a$, $x\mapsto y$, $y\mapsto z$, $z\mapsto x$.
Let $\mbox{Sub}_\tau$ be  the associated set of finite words given
by
$$\mbox{Sub}_\tau =\bigcup_{s\in \mathcal{A}, n\in \NN\cup\{0\} } \mbox{Sub}(\tau^n
(s)),$$ where $\mbox{Sub}(w)$ denotes the set of finite subwords of
the string $w$.

\begin{Theorem}[Complexity Theorem]\label{thm-complexity}
The complexity function $\complexity$ of the subshift associated with the substitution $\tau$ satisfies $\complexity (1) =
4, \complexity (2) = 6, \complexity (3) = 8$ and, for any $n\geq
2$ and $L = 2^n + k$ with $0\leq k < 2^{n}$,
$$\complexity (L) =\left\{ \begin{matrix}  2^{n+1} + 2^{n-1} + 3 k :   0\leq k < 2^{n-1} \\
 2^{n+1} + 2^n + 2 k  :  2^{n-1} \leq  k < 2^n \end{matrix}\right.$$
\end{Theorem}

To prove the theorem, we determine the difference $\mathcal{C}(L) -
\mathcal{C}(L-1)$ for each natural number $L$. In fact, we obtain
even more detailed  information and determine all right-special
words of each length. Here, a word in $Sub_\tau$ is called
\textit{right-special} if it can be extended  in more than one way
to the right by a letter of the alphabet to form another word in
$Sub_\tau$.

\begin{Theorem}[Right-special words]\label{thm-right}
Consider  $n\geq 2$  and $L = 2^n + k$ with $0\leq k < 2^{n}$.
\begin{itemize}
\item  If $0\leq k < 2^{n-1}$, then
 there exist among words of length $L$ in $Sub_\tau$, exactly two right-special words:
   the suffix of length $L$ of the word $\tau^{n}(a)$, which can be extended by
 $x,y,z$; and the suffix of length $L$ of the word $\tau^{n-2}(a) \tau^{n-2} (x)  \tau^{n-1}(a)$ which can be
 extended by $\tau^{n-2} (x)$ and by $\tau^{n-1} (x)$.

\item  If $2^{n-1}\leq k < 2^{n}$, then
 there exists a unique right-special word of length $L$, namely,
 the suffix of $\tau^{n}(a)$ of length $L$ which can be extended by
 $x,y,z$.
\end{itemize}
\end{Theorem}

\medskip

\bigskip

\medskip

%\section{The substitution $\tau$,  its subshift
%$(\varOmega_\tau,T)$ and the associated set of  finite words
%$\mbox{Sub}_\tau$}\label{The-substitution}

Recall that we consider the alphabet
$\mathcal{A} = \{a,x,y,z\}$ and denote by  $\tau$ be the substitution
mapping $a\mapsto a x a$, $x\mapsto y$, $y\mapsto z$, $z\mapsto x$.
As above, $\mbox{Sub}_\tau$ denotes the associated set of finite words given
by
$$\mbox{Sub}_\tau =\bigcup_{s\in \mathcal{A}, n\in \NN\cup\{0\} } \mbox{Sub}(\tau^n
(s)).$$ The following three properties obviously hold:

\begin{itemize}

\item The letter $a$ is a prefix of $\tau^n (a)$ for all $n\in
\NN\cup\{0\}$.

\item The lengths $|\tau^n (a)|$ converge to $\infty$ for $n\to
\infty$.

\item Every letter of $\mathcal{A}$ occurs in $\tau^n(a)$ for some  $n$.

\end{itemize}

By the first two properties, $\tau^n (a)$ is a prefix of $\tau^{n+1}
(a) $ for all $n\in\NN\cup\{0\}$. Thus, there exists a unique
one-sided infinite word $\eta$ such that $\tau^n (a)$ is a prefix of
$\eta$ for all $n\in\NN\cup\{0\}$. This $\eta$ is a fixed
point of $\tau$,
% the limit $$\omega_{\tau}:=\lim_{n\to \infty} \tau^n (a)\in\mathcal{A}^{\NN}$$
% exists.  Here, the limit means that $\tau^n(a)$ is a prefix of
% $\omega_{\tau}$ for an $n\in \NN$. This $\omega_{\tau}$ is a fixed
% point of $\tau$
 i.e., $\tau (\eta) = \eta$. We will refer to it as \textit{the fixed
 point of the substitution $\tau$}.
By the third property, we  have
$$\mbox{Sub}_\tau = \mbox{Sub} (\eta).$$
We can now associate to $\tau$ the set
$$\varOmega_\tau:=\{ \omega\in \mathcal{A}^\ZZ : \mbox{Sub} (\omega)\subset
\mbox{Sub}_\tau\}.$$  The study of $\varOmega_\tau$ can be based on
the investigation of the $\tau^n (a)$.
%We set
%$$p^{(0)}:= a\;\:\mbox{and}\;\:
% p^{(n)} :=\tau^n (a)\;\: \mbox{for}\;\: n\in \NN.$$

A direct calculation gives the following recursion formula
$$ (RF)\;\:\hspace{1cm}  \tau^{n+1}(a) =  \tau^{n}(a)  \tau^n (x)  \tau^{n}(a)\;\:\mbox{with}\;\:  \tau^n (x) =\left\{ \begin{array}{ccc}  x &:& n = 3 k, k\in \NN\cup \{0\} \\ y &:& n = 3k+1, k\in \NN\cup\{0\}\\
z&:& n = 3 k + 2, k \in \NN \cup\{0\}  \end{array}, \right. $$ It is
not hard to see that any other letter of $\eta$ is the letter $a$,
and $\eta$ can be written as
$$\eta = a r_1 a r_2 a ....$$
with a unique sequence $(r_n)$  in $\{x,y,z\}$. As $\eta$ is a fixed point of $\tau$,
this gives $$\eta = \tau^n (\eta) = \tau^{n}(a) \tau^n (r_1) \tau^{n}(a)
\tau^n (r_2) \tau^{n}(a) ...$$ This way of writing  $\eta$ is the basis
for our further analysis. While we do not need it here, we mention
in passing that the sequence $\eta$ can be generated by an automaton. The
automaton is given in Figure \ref{automaton} (see \cite{GLN} for
details).

\begin{figure}[h!]
\begin{center}
\hspace*{-1cm}
\includegraphics[scale=0.3]{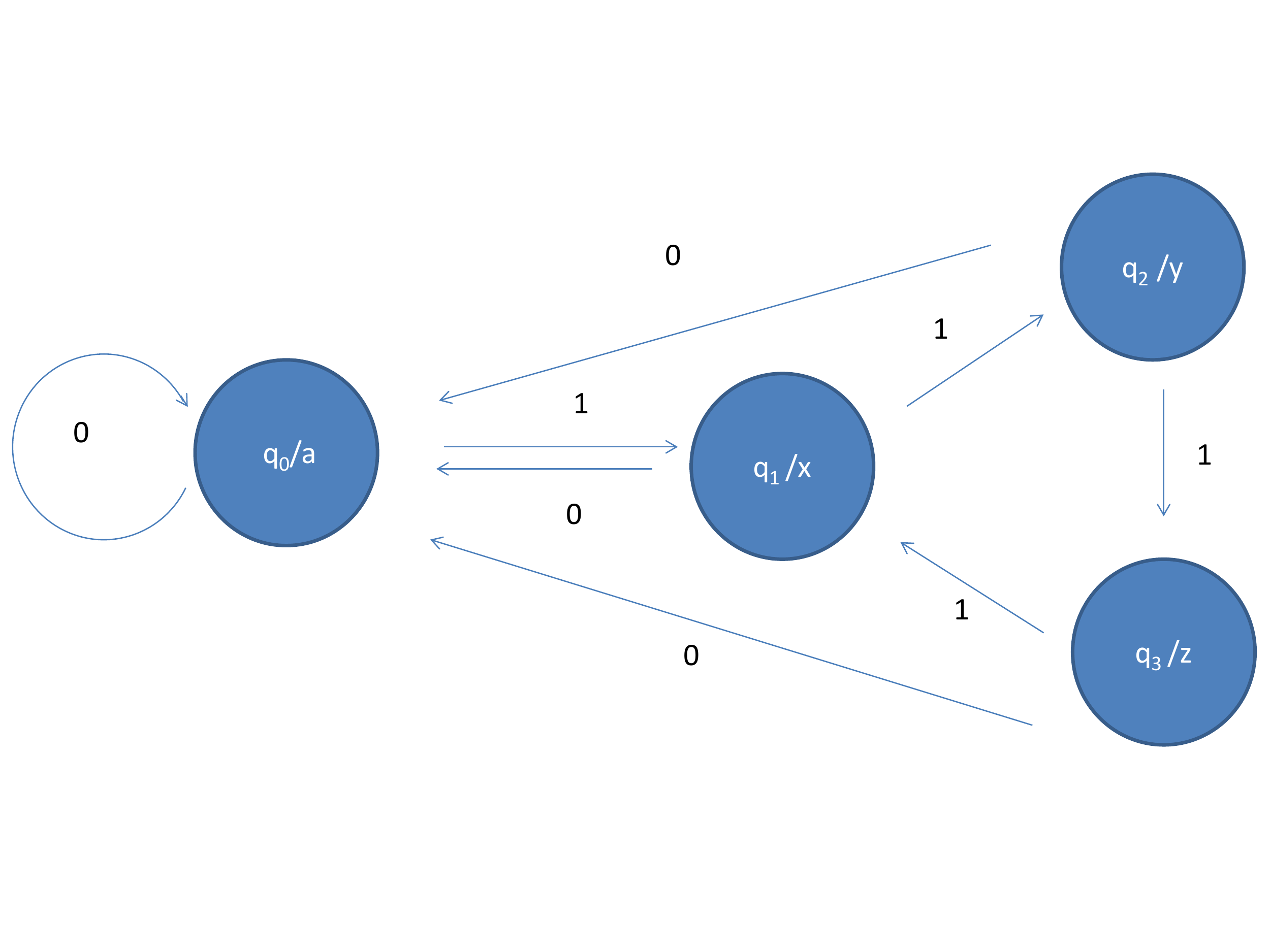}
\hspace*{-1cm}
 \caption{The automaton generating $\eta$}
\label{automaton}
\end{center}
\end{figure}

\smallskip

We define the \textit{word complexity} of $(\varOmega_\tau,T)$ as
$$\complexity : \NN\cup \{0\}\longrightarrow \NN\;\: \; \complexity (L) =
\mbox{number of elements of $\mbox{Sub}_\tau$ of length $L$}.$$ The
word $w s$ with $s\in \{a,x,y,z\}$ and $w s \in \mbox{Sub}_\tau$
will be called an \textit{extension} of $w \in \mbox{Sub}_\tau$  and
we will say in this case that $w$ can be \textit{extended by $s$}. A  word $w\in
\mbox{Sub}_\tau$ is called \textit{right-special} if the set of its
extensions has more than one element. \textit{Left-special} words
are defined in a similar way.

\begin{Proposition} \label{prop-appearance}
Any word $w\in \mbox{Sub}_\tau$ with $|w|\leq |\tau^{n}(a)| = 2^{n+1} -
1$ (for some $n\in\NN\cup\{0\}$) appears in $\tau^{n+3}(a)$.
\end{Proposition}
\begin{proof} Recall that $\mbox{Sub}_\tau = \mbox{Sub} (\eta)$.
As discussed above,  there exist $r_j^{(n)} := \tau^n (r_j) \in
\{x,y,z\}$, with
$$\eta = \tau^{n}(a) r_1^{(n)} \tau^{n}(a) r_2^{(n)} \tau^{n}(a)... $$ Thus,  any word
of length $L \leq |\tau^{n}(a)|$ is a subword of one of the three
words
$$\tau^{n}(a) x \tau^{n}(a), \tau^{n}(a) y \tau^{n}(a), \tau^{n}(a) z \tau^{n}(a).$$ These
three  words can easily be seen to appear in $\tau^{n+3}(a)$.
\end{proof}

We can use the previous proposition to obtain  the values of
$\complexity (L)$ for small values of $L$ by inspection of $\tau^{k}(a)$ for
suitable small $k$s. This gives
$$\complexity (1) = 4, \complexity (2) = 6, \complexity (3) = 8, \complexity (4) = 10.$$
From the previous proposition we also directly  obtain an upper
bound for the word  complexity.

\begin{Lemma}[Upper bound]\label{lem-upper} Let $L = |\tau^{n}(a)| = 2^{n+1} - 1$ for some $n\in
\NN$. Then,
$$\complexity (L) \leq 2 L + \frac{ L + 1}{2} = 2^{n+2} + 2^n  - 2.$$
\end{Lemma}
\begin{proof}
By the previous proposition, it suffices to give an upper bound for
the number of subwords  of length $L$  in $\tau^{n+3}(a)$.  In order
to be specific, we assume that $n$ is divisible by three. (The other
cases can be treated analogously.)  Thus, we obtain
$$\tau^{n+3}(a) = \tau^{n}(a) x \tau^{n}(a) y \tau^{n}(a) x \tau^{n}(a) z \tau^{n}(a) x \tau^{n}(a) y \tau^{n}(a) x
\tau^{n}(a).$$ Here, the $z$ in the 'middle' is at position
$|\tau^{n+2}(a)| +1$.  We will count the words of length $L$
appearing in $\tau^{n+3}(a)$ starting from the left and dismissing
words we obviously have already encountered : this will give us the
desired upper bound. We note that all subwords  of length $L$ or
less starting after the $z$ at the position $|\tau^{n+2}(a)| +1$
must have already appeared to the left of this position. Thus, we
can focus on subwords appearing in the first
$$|\tau^{n+2}(a)| + 1 = 4 L + 4$$
positions. We then see  the following:
\begin{itemize}
\item The subword of length $L$ appearing at position $\tau^{n}(a)+2 $ is
$\tau^{n}(a)$ and has already appeared at the first position.
\item The subwords of length $L$ at the positions
$P\in \left[|\tau^{n+1}(a)| +2,|\tau^{n+1}(a)| + |\tau^{n}(a)| +
3\right]$ are subwords of $\tau^{n}(a) x \tau^{n}(a)$, which have
already occurred in the prefix $\tau^{n}(a) x \tau^{n}(a)$ of
$\tau^{n+3}(a)$.
\item  There is the word  $v=\tau^{n}(a) z \tau^{n}(a)$ appearing at position
$|\tau^{n+2}(a)| + 1$ of $\tau^{n+3}(a)$. We can decompose this as
$$v=\tau^{n-1}(a) z \tau^{n-1}(a) z \tau^{n-1}(a) z \tau^{n-1}(a).$$
Then, $v$ starts with three copies of $\tau^{n-1}(a) z$. Clearly,
the subwords of length $L$ starting in the second copy of
$\tau^{n-1}(a) z$ have already appeared starting in the first copy
of $\tau^{n-1}(a) z$.
\end{itemize}
Taking these double occurrences into account we  obtain
$$\complexity (L) \leq 4 L + 4 - \left( 1 +  (|\tau^{n}(a)| + 2) +
(|\tau^{n-1}(a)| + 1)\right) = 2 L + \frac{L+1}{2}.$$ This finishes the
proof.
\end{proof}

We now complement this by a lower bound on the complexity
difference.

\begin{Lemma}[Lower bound] \label{lem-lower-cd}
(a) The inequality $\complexity (L+1) - \complexity (L) \geq 2$
holds for all $L \in \NN$.

(b) For  $L\in\NN$ with  $2^n \leq L \leq 2^n + 2^{n-1} -1$ (for
some $n\geq 2$) $\complexity (L+1) - \complexity (L) \geq 3$ holds.
\end{Lemma}
\begin{proof} (a) It suffices to show that there exists a word $w$
of length $L$ such that $w x, wy, wz$ all appear in
$\mbox{Sub}_\tau$.  Consider an arbitrary  $k\in\NN$ with $L\leq
|\tau^{k}(a)|$. It is not hard to see that $\tau^{k}(a) x, \tau^{k}(a)y,
\tau^{k}(a)z$ all appear in $\mbox{Sub}_\tau$. Thus,   $w$ can be chosen
as the suffix of $\tau^{k}(a)$ of length $L$.

(b) By assumption, we have
$$ 2^n = |\tau^{n-1}(a)| + 1 \leq L \leq |\tau^{n-1}(a)| + 1 + |\tau^{n-2}(a)| =
2^n + 2^{n-1} -1.$$ In order to be specific, let us assume that
 $n$ is
such that $\tau^{n}(a) = \tau^{n-1}(a) x \tau^{n-1}(a)$. As already noted
in the proof of (a), any  suffix of $\tau^{n}(a)$ has three different
extensions. This is in particular true for the suffix of $\tau^{n}(a)$
of length $L$. Note that this suffix  has $x \tau^{n-1}(a)$ as a suffix
(due to $L\geq |\tau^{n-1}(a)| + 1$).

We are going to find another word with two different extensions.
Considering $\tau^{n+3}(a)$, we find that $w = \tau^{n}(a) z \tau^{n}(a) y$
belongs to $\mbox{Sub}_\tau$. Now, a short calculation gives
$$w = \tau^{n-1}(a) x \tau^{n-2}(a) z \tau^{n-2}(a) z \tau^{n-2}(a) z \tau^{n-2}(a) x
\tau^{n-1}(a) y.$$ Thus,
$$ \tau^{n-2}(a) z \tau^{n-2}(a) z \tau^{n-2}(a) z  = \tau^{n-2}(a) z
\tau^{n-1}(a) z $$ belongs to $\mbox{Sub}_\tau$. This shows that every
suffix of $u$ of $\tau^{n-2}(a) z \tau^{n-1}(a)$ can be extended by $z$. On
the other hand, the above formula also gives that $\tau^{n-1}(a) z
\tau^{n-1}(a) x$ occurs in $w$. Thus, $u$  can also be extended by $x$
and, hence, has two  extensions. Taking into account that the
suffices of length $L$ of $\tau^{n-2}(a) z \tau^{n-1}(a)$ have  $z
\tau^{n-1}(a)$ as a suffix, (which is different from $x \tau^{n-1}(a)$,) we
see that these extensions are different from the previously
encountered extensions.
\end{proof}
Clearly, the previous lemma  implies a  lower bound on the
complexity function via
$$\complexity (L) \geq \complexity (L_0) + \sum_{k=L_0}^{L-1} (\complexity (k+1) -
\complexity (k))$$ for arbitrary $L_0\leq L$ in $\NN$. We are going
to use this for a particular $L$ and then for $L_0$.

\begin{Proposition}\label{lem-sometimes-sharp} For $L$ with $L = |\tau^{n}(a)| = 2^{n+1} - 1$ with $n\geq 2$,
we have $\complexity (L) = 2^{n+2} + 2^n - 2.$
\end{Proposition}
\begin{proof}
We decompose the interval $[4, 2^{n+1}-1] $  according to powers of
$2$ as
$$\bigcup_{k=1}^{n-2} \left( [2^{k+1}, 2^{k+1} + 2^k-1]
\cup [2^{k+1} + 2^k-1, 2^{k+2}]\right) \cup ( [2^{n}, 2^{n} +
2^{n-1}-1] \cup [2^{n} + 2^{n-1}, 2^{n+1}-1]  ).$$ We then apply (a)
and (b) of the previous lemma to obtain
$$ \complexity (L) \geq   \complexity (4) -2 + \sum_{k=1}^{n-1}  (3  \cdot 2^k + 2 \cdot 2^k)  = 2^{n+2} + 2^{n}
-2.$$
 Here,  $-2$ arises in the first step, as our  sum
deals with  the full interval $[4, 2^{n+1}]$, whereas we actually
only need the interval $[4, 2^{n+1}-1]$. In the last  step, we use
the already checked $\complexity (4) = 10$. Comparing this
lower bound with the upper bound of Lemma \ref{lem-upper}, we see
that both bounds agree, and this gives the desired statement.
\end{proof}

After these preparations we can now state the main technical
ingredient of our investigation.

\begin{Lemma}[Growth of complexity] \label{growth-lemma}
For any $n\geq 2$,  and for $L = 2^n + k$ with $0\leq k < 2^{n}$, we have
$$
\complexity (L+1) - \complexity (L)=\left\{ \begin{matrix}  3 :
0\leq k < 2^{n-1} \\ 2 :  2^{n-1} \leq k < 2^n \end{matrix}\right.$$
\end{Lemma}
\begin{proof} It suffices to show that for any $n\in \NN$, and for $L =
2^{n+1} - 1$,  the  lower bound on the complexity function  coming
from Lemma \ref{lem-lower-cd} gives the correct value of the
complexity function. This is, however, just the content of the
previous proposition. This finishes the proof.
\end{proof}

Based on this lemma we can now determine, in Theorem \ref{thm-complexity} and in Theorem \ref{thm-right}, the complexity and
the set of right-special words.

\begin{proof}[Proof of Theorem \ref{thm-complexity}]  This  follows easily from Lemma \ref{growth-lemma} and
Proposition \ref{lem-sometimes-sharp}.
\end{proof}

\begin{proof}[Proof of Theorem \ref{thm-right}] From Lemma \ref{growth-lemma} we   know that the lower bounds in Lemma \ref{lem-lower-cd}
are sharp. Inspecting the proof of that lemma, we find exactly
the  words which   have more than one extension.
\end{proof}

\bigskip

\noindent \textbf{Remark.}

\medskip

 (a) We note that $\mbox{Sub}_\tau$ is closed under
reflections, hence, the statements in Theorem \ref{thm-right} about
right-special words easily translate on corresponding statements
about left-special words. As a consequence one obtains  that the
words $\tau^{n}(a)$, $n\in \NN\cup \{0\}$, are both right-special
and left-special (and are the only words with this property).

\medskip

(b)  Similar subshifts of low complexity can be associated to other
groups of intermediate growth. This more general framework and other
related topics  will be discussed in a forthcoming paper.

\medskip

(c) To any subshift one can associate a family of graphs, indexed by
the natural numbers,  known as Rauzy graphs (or de Bruijn graphs in
the case of the full shift) in the following way:  For each natural
number $n$ the vertices are the words of length $n$ and the edges
are the words of length $n+1$. Specifically, there is an edge from
the  $u$ to $v$ if there exists a word $w$ of length $n+1$ such that
$u$ is a prefix of of $w$ and $v$ is a suffix of $w$.

Then, these graphs basically store the information on the
right-special words. Thus, given that we have determined these
words, it is possible to determine the Rauzy graphs based on the
material presented in this section, compare Figure \ref{Rauzy}.
Further details will be provided elsewhere.

\begin{figure}[h!]
\begin{center}
\hspace*{-1cm}
\includegraphics[scale=0.3]{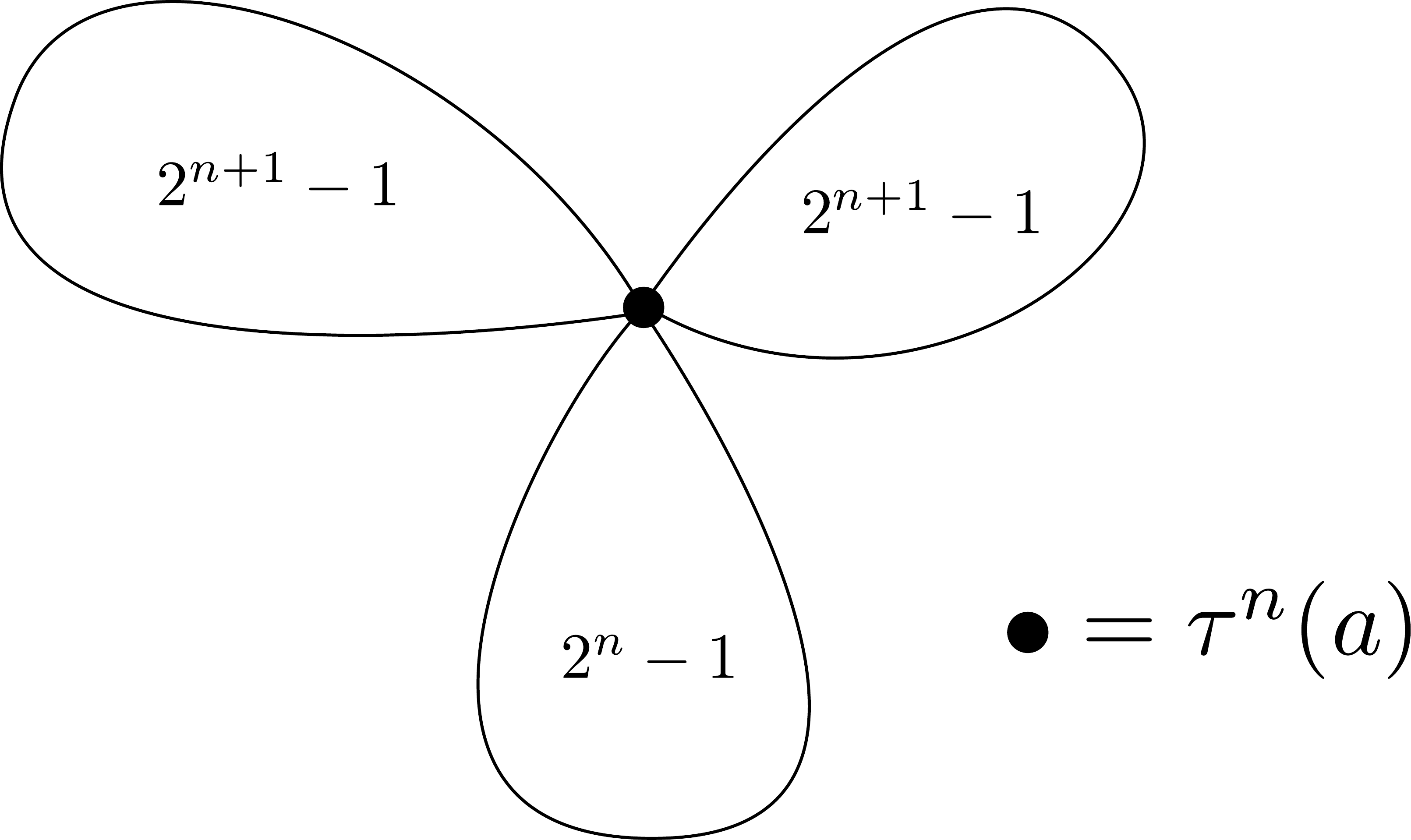}
\hspace*{-1cm}
 \caption{The Rauzy graph containing $\tau^n (a)$. The numbers give the number of vertices for each loop (without counting the middle vertex).}
\label{Rauzy}
\end{center}
\end{figure}

\bigskip

\medskip

\textbf{Acknowledgments.} R. G. was partially supported by the NSF
grant DMS-1207669 and by ERC AG COMPASP. The authors acknowledge
support of the Swiss National Science Foundation. Part of this
research was carried out while D. L. and R. G. were visiting the
Department of mathematics of the University of Geneva. The
hospitality of the department is gratefully acknowledged. The
authors thank Roman Kogan for providing Figure 2.


\begin{thebibliography}{99}
\bibitem{BGr} M.~Baake, U.~Grimm: \textit{Aperiodic Order: Volume 1, A Mathematical Invitation},
Encyclopedia of Mathematics and its Applications \textbf{149}, Cambridge university press, Cambridge (2014).

%\bibitem{BM} M.~Baake, R.~Moody (Eds.): \textit{Directions in Mathematical Quasicrystals},
%CRM Monogr.\ Ser.\ {\bf 13}, Amer.\ Math.\ Soc., Providence, RI (2000).

%\bibitem{BG} L.~Bartholdi and R.~I.~Grigorchuk:
%\textit{On the spectrum of Hecke type operators related to some fractal groups},  Tr. Mat. Inst. Steklova
%{\bf 231} (2000), Din. Sist., Avtom. i Beskon. Gruppy, 5--45; translation in {\it Proc. Steklov Inst. Math.} (2000), no. 4 ({\bf 231}), 1--41.


%\bibitem{DL2} D.~Damanik, D.~Lenz: \textit{Substitution dynamical systems: characterization of linear repetitivity and applications},  J. Math. Anal. Appl. \textbf{321} (2006),  766--780.

%\bibitem{DZ} D.~Damanik, D.~Zare: \textit{Palindrome complexity bounds for primitive substitution sequences}, Disc. Math. \textbf{222} (2000), 259--267.


%\bibitem{Dur} F.~Durand: \textit{Linearly recurrent subshifts have a finite number of non-periodic subshift factors}, Ergod.\ Th.\ \& Dynam.\ Sys. {\bf 20} (2000), 1061--1078.


%\bibitem{DHS} F.~Durand, B.~Host, C.~Skau: \textit{Substitution dynamical systems, Bratteli diagrams and dimension groups}, Ergod.\ Th.\ \& Dynam.\ Sys. {\bf 19} (1999), 953--993.


\bibitem{Gri80} R.~I.~Grigorchuk: \textit{On Burnside's problem on periodic groups. (Russian)}
Funktsional. Anal. i Prilozhen. \textbf{14} (1980), 53 -- 54.

\bibitem{Gri84} R.~I.~Grigorchuk: \textit{Degrees of growth of finitely generated groups and the theory of invariant means (Russian)},
Izv. Akad. Nauk SSSR Ser. Mat. \textbf{48} (1984),  939--985.


%\bibitem{Gri11} R.~I.~Grigorchuk:
%\textit{Some problems of the dynamics of group actions on rooted trees (Russian)}, Tr. Mat. Inst. Steklova \textbf{273} (2011), Sovremennye Problemy Matematiki, 72--191; translation in Proc. Steklov Inst. Math. \textbf{273} (2011),  64--175.



\bibitem{GLN} R.~I.~Grigorchuk, D.~Lenz, T.~Nagnibeda:
\textit{Spectra of Schreier graphs of Grigorchuk's group and
Schr\"odinger operators with aperiodic order},  arXiv:1412.6822.


\bibitem{GLN-survey} R.~I.~Grigorchuk, D.~Lenz, T.~Nagnibeda:
\textit{Schreier graphs of Grigorchuk's group and a subshift
associated to a non-primitive Substitution}, to appear in: Groups,
Graphs, and Random Walks. T. Ceccherini-Silberstein, M. Salvatori
and E. Sava-Huss Eds, London Mathematical Society Lecture Note
Series, Cambridge University Press, Cambridge, 2017.


%\bibitem{GLNS} R.~I.~Grigorchuk, D.~Lenz, T.~Nagnibeda, D.~Sell:
%\textit{Spinal groups and Toeplitz subshifts}, in preparation.


\bibitem{KLS} J.~Kellendonk, D.~Lenz, J.~Savinien (eds): \textit{Directions in aperiodic order},
Progress in Mathematics \textbf{309}, Birkh\"auser, Basel (2015).



%\bibitem{LP} J.~Lagarians,  P. A. B. Pleasants: \textit{Repetitive Delone sets and quasicrystals}, Ergod.\ Th.\ \& Dynam.\ Sys.  {\bf 23} (2003), 831--867.


%\bibitem{LM} D.~Lind, B.~Marcus:  \textit{An introduction to symbolic dynamics and coding}, Cambridge University Press, Cambridge, (1995).



\bibitem{Lys} I.~G.~ Lysenok: \textit{A set of defining relations for the Grigorchuk group (Russian)},
Mat. Zametki \textbf{38} (1985),  503--516. English translation in:
Math. Notes \textbf{38} (1985), 784--792.

\bibitem{MB} N.~Matte-Bon: \textit{Topological full groups of minimal subshifts with subgroups of intermediate growth}, Journal of Modern Dynamics, {\bf 9} (2015),  67--80.

\bibitem{Nek} V.~Nekrashevych: \textit{Palindromic subshifts and simple periodic groups of intermediate growth}, Preprint 2016 (arXiv:1601.01033).

\bibitem{Wal} P.~Walters, \textit{An Introduction to Ergodic Theory}, Springer, New York (1982).



\end{thebibliography}
\end{document}